\documentclass[14pt]{article}
\usepackage{amsmath}
\usepackage{amsfonts}

\newtheorem{theorem}{Theorem}

\newtheorem{remark}[theorem]{Remark}

\newenvironment{proof}[1][Proof]{\noindent\textbf{#1.} }{\ \rule{0.5em}{0.5em}}
\numberwithin{equation}{section}

\begin{document}

\title{L\'{e}vy processes linked to the lower-incomplete gamma function }
\author{Luisa Beghin\thanks{%
Sapienza University of Rome, P.le A. Moro 5, 00185 Roma, Italy. e-mail:
\texttt{luisa.beghin@uniroma1.it}} \ \ \ \ \ Costantino Ricciuti\thanks{%
Sapienza University of Rome, P.le A. Moro 5, 00185 Roma, Italy. e-mail:
\texttt{costantino.ricciuti@uniroma1.it}} }
\date{}
\maketitle

\begin{abstract}
We start by defining a subordinator by means of the lower-incomplete gamma
function. It can be considered as an approximation of the stable
subordinator, easier to be handled thank to its finite activity. A tempered
version is also considered in order to overcome the drawback of infinite
moments. Then, we study L\'{e}vy processes time-changed by these
subordinators, with particular attention to the Brownian case. An
approximation of the fractional derivative (as well as of the fractional
power of operators) arises from the analysis of governing equations.
Finally, we show that time-changing the fractional Brownian motion gives a
model of anomalous diffusion, which exhibits a sub-diffusive behavior.
\vspace{0.5cm}

\end{abstract}

\section{Introduction}

In the spirit of \cite{BEG}, we consider here a subordinator $S_{\alpha
}(t),t\geq 0,$ defined by means of the lower-incomplete gamma function of
parameter $\alpha \in (0,1]$, i.e.
\begin{equation}
\gamma (\alpha ,x)=\int_{0}^{x}e^{-w}w^{\alpha -1}dw,\qquad x>0.
\label{incomplete gamma function}
\end{equation}%
We will see that, in the special case $\alpha =1,$ it reduces to a
homogeneous Poisson process, while, in general, it can be represented as a
compound Poisson process with positive jumps of size greater than one. Such
a process retains many properties of the stable subordinator, e.g. the tail
behavior of the distribution and the asymptotic form of the fractional
moments, even if it loses the property of self-similarity. A standard
reference for the theory of stable processes is \cite{SAM}.

By a slight modification we are led to a new subordinator whose jumps are
greater than $\epsilon >0$, which converges to a stable one in the limit for
$\epsilon \rightarrow 0$. We prove that its density $q_{\epsilon }(x,t)$
solves an equation where a perturbation of the Riemann fractional derivative
appears. When $\epsilon \rightarrow 0$, such operator reduces to the Riemann
derivative and we obtain the well known equation governing the stable
density.  For an introduction to fractional derivatives and fractional
equations consult \cite{MEE}.

The above framework can be extended to the so-called multivariate
subordinators, i.e. multidimensional L\'evy processes with increasing
marginal components (for their properties and applications see e.g. \cite%
{sato} and \cite{BEG2}).

In order to overcome the drawback of infinite moments of $S_{\alpha }$, we
consider a tempered version of our subordinator, say $S_{\alpha, \theta}(t),
t\geq 0$, where $\theta\geq 0$ is the tempering parameter, whose
distribution displays finite moments of any integer order.

We use these subordinators as independent random times of well-known L\'evy
processes. As for other subordinated processes already studied in the
literature, the time-change allows to keep some properties of the external
process and simultaneously modify some other features (consult \cite{SAT}
for the general theory).

When considering the process $B(S_{\alpha ,\theta }(t)),$ $t\geq 0,$ where $%
B:=\{B(t),$ $t\geq 0\}$ is a standard Brownian motion and $S_{\alpha ,\theta
}$ is supposed independent from $B$, we obtain the following auto-covariance
function
\begin{equation*}
Cov\left( B(S_{\alpha ,\theta }(t)),B(S_{\alpha ,\theta }(\tau ))\right)
=\alpha (t\wedge \tau )\theta ^{\alpha -1}e^{-\theta },\qquad t,\tau \geq 0.
\end{equation*}%
Even if it is linear in the time argument, as for the standard Brownian
motion, the parameters $\alpha $ and $\theta $ model the deviation from the
dependence structure of $B$: in particular, for $\theta \rightarrow 0$ and
for $\alpha $ strictly less than $1$, the auto-covariance tends to infinity,
for any $t.$

Finally, we consider a fractional Brownian motion subordinated by $S_{\alpha
}(t)$ (for basic notions on the fractional Brownian motion consult e.g. \cite%
{MAN}). We show that the model obtained still displays the long-range
dependence, with a rate depending not only on the Hurst index $H$, but also
on $\alpha .$ Moreover, it is proved to behave
asymptotically as a subdiffusion, according to the value of the parameter $%
\alpha $: the subdiffusive behavior is more marked the greater the value of $%
\alpha $ (for any fixed $H$). We recall that a process is said to be
subdiffusive if, for large times $t$, the mean square displacement grows as $%
t^\gamma$ with $\gamma <1$. We refer to \cite{MET} for an overview on
anomalous diffusion models and their applications.

\section{Basic notions and preliminary results}


We recall the following definition: a function $\varphi :(0,\infty
)\rightarrow \mathbb{R}$ is a Bernstein function if $\varphi $ is of class $%
C^{\infty },$ $\varphi (\eta )\geq 0,$ for any $\eta > 0$, and
\begin{equation}
(-1)^{n-1}\frac{d^{n}}{dx^{n}}\varphi (\eta )\geq 0,  \label{b1}
\end{equation}%
for any $n\in \mathbb{N}$ and $\eta >0.$ It is well known that any Bernstein
function $\varphi $ admits the following representation
\begin{equation}
\varphi (\eta )=a+b\eta +\int_{0}^{+\infty }(1-e^{-s\eta })\nu (ds),
\label{lk}
\end{equation}%
for $a,b\geq 0$ and where $\nu (\cdot )$ denotes a measure on $%
(0,+\infty )$ such that%
\begin{equation*}
\int_{0}^{+\infty }(s\wedge 1)\nu (ds)<\infty .
\end{equation*}%
The triplet $(a,b,\nu )$ is called the L\'{e}vy triplet of the Bernstein
function $\varphi $ (see, for example, \cite{SCH}, p.21) and $\nu (\cdot )$
is a L\'{e}vy measure.

Finally, a Bernstein function $\varphi $ is complete if and only if its L%
\'{e}vy measure in (\ref{lk}) has a completely monotone density $m(\cdot )$
with respect to the Lebesgue measure, i.e.the following representation holds
for a completely monotone function $m(\cdot )$:
\begin{equation}
\varphi (\eta )=a+b\eta +\int_{0}^{+\infty }(1-e^{-s\eta })m(s)ds.
\end{equation}

\subsection{Univariate subordinators}

\label{sottosezione univariate}

We now recall that a subordinator $S(t),$ $t\geq 0$, is a L\'{e}vy process
with non-decreasing paths and that, for any Bernstein function $\varphi $,
 there exists a subordinator $S(t)$ such that
\begin{equation*}
\mathbb{E}e^{-\eta S(t)}=e^{-t\varphi (\eta )}
\end{equation*}%
(see, for example, \cite{BER} and \cite{APP}). In the special case where $%
\varphi (\eta )=\eta ^{\alpha }$, for $\alpha \in (0,1),$ it is well-known
that $S:=\{S(t),t\geq 0\}$ is a $\alpha $-stable subordinator and its
density satisfies the following equation:
\begin{equation}
\frac{\partial }{\partial t}h(x,t)=-\frac{\partial ^{\alpha }}{\partial
x^{\alpha }}h(x,t),\qquad h(x,0)=\delta (x),  \label{riemann}
\end{equation}%
for $x,t\geq 0,$ where $\frac{\partial ^{\alpha }}{\partial x^{\alpha }}$ is
the Riemann-Liouville fractional derivative of order $\alpha ,$ defined  as%
\begin{equation*}
\frac{\partial ^{\alpha }}{\partial x^{\alpha }}f(x)=\left\{
\begin{array}{l}
\frac{1}{\Gamma (1-\alpha )}\frac{d}{dx}\int_{0}^{x}\frac{f(t)}{%
(x-t)^{\alpha }}dt,\qquad \alpha \in (0,1) \\
\frac{d}{dx}f(x),\qquad \alpha =1%
\end{array}%
\right. ,
\end{equation*}%
for a locally integrable function $f$ on $(0,+\infty )$ (see \cite{KIL},
p.70)$.$ This can be easily checked by considering formula (2.2.36) in \cite%
{KIL} and applying the Laplace transform to both members of (\ref{riemann}),
which gives $\widetilde{h}(\eta ,t)=e^{-t\eta ^{\alpha }}.$

\subsection{Multivariate subordinators}

\label{sottosezione multivariate}

In the multivariate case, we recall that a subordinator in the sense of \cite%
{sato} and \cite{BEG2} is a $d$-dimensional L\'{e}vy process with increasing
marginal components. We denote a multivariate subordinator by
\begin{equation*}
(S_{1}(t),S_{2}(t),\dots ,S_{d}(t)).
\end{equation*}%
The multivariate L\'{e}vy measure $\nu (dx_{1},\dots ,dx_{d})$ satisfies the
following condition
\begin{equation*}
\int_{\mathbb{R}_{+}^{d}}\min \biggl (1,\sqrt{x_{1}^{2}+\dots +x_{d}^{2}}%
\biggr)\,\nu (dx_{1},\dots ,dx_{d})<\infty ,
\end{equation*}%
where $\mathbb{R}_{+}^{d}=\{(x_{1},\dots ,x_{d})\in \mathbb{R}^{d}:x_{1}\geq
0,x_{2}\geq 0,\dots ,x_{d}\geq 0\}$.

Its $d$-dimensional Laplace transform reads
\begin{equation*}
\mathbb{E}e^{-(\eta _{1}S_{1}(t)+\eta _{2}S_{2}(t)+\dots +\eta
_{d}S_{d}(t))}=e^{-t\Phi (\eta _{1},\dots ,\eta _{d})},\qquad \eta _{1}\geq
0\dots \eta _{d}\geq 0,
\end{equation*}%
where
\begin{equation*}
\Phi (\eta _{1},\dots ,\eta _{d})=\int_{\mathbb{R}_{+}^{d}}\left[
1-e^{-(\eta _{1}x_{1}+\dots +\eta _{d}x_{d})}\right] \nu (dx_{1},\dots
,dx_{d})
\end{equation*}%
is a multivariate Bernstein function.

A $d$-dimensional subordinator is said to be stable if, using the spherical
variables $\rho \in (0,\infty )$ and $\theta \in B^{d-1}$ ($B^{d-1}$
denoting the $d-1$-dimensional unit sphere), its L\'{e}vy measure can be
expressed as

\begin{equation*}
\nu (d\rho ,d\theta )=C\rho ^{-\alpha -1}M(d\theta ),
\end{equation*}%
where $M(d\theta )$ is a probability measure on $B_{+}^{d-1}=B^{d-1}\cap
R_{+}^{d}$ . In other words, a $d$-dimensional stable subordinator\ is a
multivariate stable process with increasing marginal components.

In this case, the Bernstein function reads
\begin{equation*}
\Phi (\eta )=k\int_{B_{+}^{d-1}}(\theta \cdot \eta )^{\alpha }M(d\theta
),\qquad \eta =(\eta _{1},\dots ,\eta _{d}).
\end{equation*}%
By Laplace inversion, the density $q(x,t),x\in R_{+}^{d},$ $t\geq 0$ of a
multivariate stable subordinator satisfies the following equation
\begin{equation}
\frac{\partial }{\partial t}q(x,t)=-k\int_{B_{+}^{d-1}}(\nabla _{x}\cdot
\theta )^{\alpha }q(x,t)\,M(d\theta ),  \label{dir}
\end{equation}%
where $(\nabla _{x}\cdot \theta )^{\alpha }$ is the fractional directional
derivative along the unit vector $\theta $, defined as
\begin{equation*}
(\nabla _{x}\cdot \theta )^{\alpha }h(x):=k\,\int_{0}^{\infty }\bigl (%
h(x)-h(x-r\theta )\bigr )r^{-\alpha -1}dr.
\end{equation*}%
Thus the operator on the right-hand side of (\ref{dir}), also studied in
\cite{garra} and \cite{MEE}, is the average, under the measure $M(d\theta )$%
, of $(\nabla _{x}\cdot \theta )^{\alpha }$. For $d=2$ we have $\theta
=(\cos \beta ,\sin \beta )$, and the operator takes the following form
\begin{equation*}
-k\int_{0}^{\frac{\pi }{2}}\bigl (\cos \beta \frac{\partial }{\partial x_{1}}%
+\sin \beta \frac{\partial }{\partial x_{2}}\bigr )^{\alpha
}q(x_{1},x_{2},t)\,M(d\beta ).
\end{equation*}

\subsection{Fractional equation satisfied by the incomplete gamma function}

\label{sottosezione gamma}

The incomplete Gamma function defined in (\ref{incomplete gamma function})
is a Bernstein function. Indeed it is non-negative, $C^{\infty }$, null at
the origin, with derivatives satisfying
\begin{eqnarray*}
\frac{d}{d\eta }\gamma (\alpha ;\eta ) &=&e^{-\eta }\eta ^{\alpha -1}\geq 0,
\\
\frac{d^{2}}{d\eta ^{2}}\gamma (\alpha ;\eta ) &=&-\frac{d}{d\eta }\gamma
(\alpha ;\eta )+(\alpha -1)e^{-\eta }\eta ^{\alpha -2}\leq 0, \\
\frac{d^{3}}{d\eta ^{3}}\gamma (\alpha ;\eta ) &=&-\frac{d^{2}}{d\eta ^{2}}%
\gamma (\alpha ;\eta )-(\alpha -1)e^{-\eta }\eta ^{\alpha -2}+(\alpha
-1)(\alpha -2)e^{-\eta }\eta ^{\alpha -3}\geq 0,
\end{eqnarray*}%
and so on.

Preliminarily, we show that the lower-incomplete Gamma function (\ref%
{incomplete gamma function}) solves the following integro-differential
equation
\begin{equation}
\frac{\alpha }{\Gamma (1-\alpha )}\int_{0}^{x}\frac{d}{ds}u(x-s)\Gamma
(-\alpha ,s)ds=\Gamma (\alpha )-u(x),\qquad u(0)=0, \label{eq temp}
\end{equation}%
where $\Gamma (\beta ,x)=\int_{x}^{\infty }e^{-w}w^{\beta -1}dw$ is the
upper incomplete Gamma function (which is defined for any $\beta ,x\in
\mathbb{R}$ and is real-valued for $x\geq 0$). Up to a multiplication by $%
\alpha $, the operator on the left-side is the Caputo fractional derivative
with tempered kernel (see, e.g., \cite{KUM3}). We observe that (\ref{eq temp}%
) is a relaxation equation because the solution $u(x)=\gamma (\alpha ,x)$
converges to the stationary solution $\tilde{u}(x)=\Gamma (\alpha )$ as $%
x\rightarrow \infty $.

Let now $u:\mathbb{R}^{+}\rightarrow \mathbb{R}^{+}$ be an absolutely
continuous function, such that $|u(x)|\leq ce^{kx},$ for some $c,k>0$ and
for any $x\geq 0,$ then we define the operator
\begin{equation}
\mathcal{D}_{t}^{\lambda ,\rho }u(t):=\frac{\rho \lambda ^{\rho }}{\Gamma
(1-\rho )}\int_{0}^{t}\frac{d}{dt}u(t-s)\Gamma (-\rho ;\lambda s)ds,\qquad
\rho \in (0,1),\;\lambda >0.  \label{temp}
\end{equation}
It has been proved in \cite{BEG} that $f(t)=\Gamma (\rho ;\lambda t)$ is the
eigenfunction of the operator $\mathcal{D}_{t}^{\lambda ,\rho },$ i.e. that $%
\mathcal{D}_{t}^{\lambda ,\rho }f=-\lambda ^{\rho }f$. Then, by recalling
that $\Gamma (\alpha ;x)+\gamma (\alpha ;x)=\Gamma (\alpha )$, the Cauchy
problem (\ref{eq temp}) is easily verified. Indeed, $\mathcal{D}%
_{t}^{\lambda ,\rho }K=1$, for any $K\in \mathbb{R},$ by (\ref{temp}) and,
moreover, $\gamma (\alpha ;\cdot )$ is absolutely continuous on $\mathbb{R}%
^{+}$ and $|\gamma (\alpha ;x)|\leq \Gamma (\alpha )\leq \Gamma (\alpha
)e^{kx},$ for any $x,k\geq 0.$

As an alternative proof, we recall that the Laplace transform of (\ref{temp}%
) is given by%
\begin{equation}
\int_{0}^{+\infty }e^{-\theta t}\mathcal{D}_{t}^{\lambda ,\rho }u(t)dt=\left[
(\theta +\lambda )^{\rho }-\lambda ^{\rho }\right] \widetilde{u}(\theta )-%
\frac{\left[ (\theta +\lambda )^{\rho }-\lambda ^{\rho }\right] }{\theta }%
u(0),\qquad \theta >0  \label{lap}
\end{equation}%
(see \cite{BEG}); moreover,
\begin{equation*}
\int_{0}^{+\infty }e^{-\theta x}\gamma (\alpha ;x)dx=\int_{0}^{+\infty
}e^{-w}w^{\alpha -1}\int_{w}^{+\infty }e^{-\theta x}dxdw=\frac{\Gamma
(\alpha )}{\theta (\theta +1)^{\alpha }},
\end{equation*}%
so that the Laplace transforms of the two sides of (\ref{eq temp}) coincide.
We can easily check that, for $\alpha =1,$ the equation (\ref{eq temp})
reduces to
\begin{equation*}
\frac{d}{dx}u(x)=1-u(x),
\end{equation*}%
which (for $u(0)=0$) is satisfied by $u(x)=1-e^{-x}=\gamma (1;x),$ even
though the expression of $\mathcal{D}_{t}^{\lambda ,\rho }$ given in (\ref%
{temp}) is not well-defined in this special case.

\section{The subordinator $S_{\protect\alpha }$}

\subsection{Definition and properties}

We start by considering the subordinator defined by means of the
lower-incomplete gamma function, i.e. with Laplace exponent $\alpha \gamma
(\alpha ;\eta ),$ for $\alpha \in (0,1]$.

\begin{theorem}
Let $\alpha \in (0,1],$ then the function%
\begin{equation}
\varphi (\eta ):=\alpha \gamma (\alpha ;\eta ), \qquad \eta \geq0  \label{b2}
\end{equation}%
is the Laplace exponent of a finite-activity (or step) subordinator $%
S_{\alpha }:=\left\{ S_{\alpha }(t),t\geq 0\right\} $, with triplet $%
(0,0,\pi ),$ where $\pi $ is an absolutely continuous L\'{e}vy measure, with
completely monotone density%
\begin{equation}
\overline{\pi }(z)=\frac{1_{z\geq 1}\alpha (z-1)^{-\alpha }z^{-1}}{\Gamma
(1-\alpha )}.  \label{b3}
\end{equation}
\end{theorem}

\begin{proof}
The incomplete gamma function $\gamma (\alpha ,x)$ is a Bernstein function,
as explained in section \ref{sottosezione gamma}. Hence also $\alpha \gamma
(\alpha ,x)$ is a Bernstein function. We now prove that representation (\ref%
{lk}) holds, in this case, for $a=b=0$ and for the L\'{e}vy measure given in
(\ref{b3}): indeed we have that%
\begin{eqnarray*}
\int_{0}^{+\infty }(1-e^{-\eta x})\pi (dx) &=&\int_{0}^{+\infty
}x\int_{0}^{\eta }e^{-zx}dz\overline{\pi }(x)dx \\
&=&\int_{0}^{\eta }dz\int_{1}^{+\infty }xe^{-zx}\frac{\alpha (x-1)^{-\alpha
}x^{-1}}{\Gamma (1-\alpha )}dx \\
&=&\int_{0}^{\eta }e^{-z}dz\int_{0}^{+\infty }e^{-zw}\frac{\alpha w^{-\alpha
}}{\Gamma (1-\alpha )}dw \\
&=&\int_{0}^{\eta }\frac{\alpha e^{-z}}{z^{1-\alpha }}dz=\alpha \gamma
(\alpha ;\eta ),
\end{eqnarray*}%
where the interchange of the integrals' order is allowed by the absolute
convergence of the double integral and the application of the Fubini
theorem. In order to prove that $S_{\alpha }$ does not have strictly
increasing trajectories, we must show that the integral of the L\'{e}vy
measure on $(0,\infty )$ is finite. Indeed by (\ref{lk}) the last condition,
together with $a=b=0,$ is sufficient to prove that a subordinator is a step
process (i.e. it has piecewise sample paths), see \cite{SAT}, p.135: in this
case we have that%
\begin{equation}
\int_{0}^{+\infty }\pi (dz)=\int_{1}^{+\infty }\frac{\alpha (z-1)^{-\alpha
}z^{-1}}{\Gamma (1-\alpha )}dz=\frac{\alpha \Gamma (1-\alpha )\Gamma (\alpha
)}{\Gamma (1-\alpha )}=\alpha \Gamma (\alpha )<\infty ,  \label{ss}
\end{equation}%
by considering formula (3.191.2) of \cite{GRA}, since $\alpha >0.$ Finally,
it is easy to check, by differentiating, that the density of the L\'{e}vy
measure in (\ref{b3}) is completely monotone.
\end{proof}

\begin{remark}
In the limiting case where $\alpha \rightarrow 1^{-}$ the process $S_\alpha$
reduces to the Poisson process. Indeed we have that $\lim_{\alpha
\rightarrow 1^{-}}\pi (dz)=\delta _{1}(z)dz$, which is the L\'{e}vy measure
of the Poisson process of rate $1$: this can be seen by considering that%
\begin{equation*}
\lim_{\alpha \rightarrow 1^{-}}\varphi (\eta )=\gamma (1,\eta )=1-e^{-\eta
}=\int_{0}^{+\infty }(1-e^{-\eta x})\delta _{1}(x)dx.
\end{equation*}
\end{remark}

We underline that the L\'{e}vy measure given in (\ref{b3}) is different from
zero only for $z\geq 1$: this means that the subordinator performs almost surely
jumps of size greater than one. As a
consequence and by considering that its diffusion coefficient is zero, the
process $S_{\alpha }$ has also finite variation (see Theorem 21.9 in \cite%
{SAT}).

Moreover, the result in (\ref{ss}) implies that $S_{\alpha }$ is a L\'{e}vy
process of type A (see Def.11.9 in \cite{SAT}, p.65) and has finite
activity, i.e. its number of jumps is finite on every compact interval for
almost all the paths (see Thm. 21.3 in \cite{SAT}). Thus $S_{\alpha }$ can
be represented as a compound Poisson process
\begin{equation}
S_{\alpha }(t)=\sum_{j=1}^{N_{\alpha }(t)}Z_{j}^{\alpha },  \label{cp}
\end{equation}%
where $N_{\alpha }:=\left\{ N_{\alpha }(t),t\geq 0\right\} $ is a
homogeneous Poisson process with rate $\lambda =\alpha \Gamma (\alpha )$ and
the jumps $Z_{j}^{\alpha }$ are i.i.d. random variables, taking values in $%
[1,+\infty )$, with probability density
\begin{equation*}
f_{Z^{\alpha }}(z)=\frac{ (z-1)^{-\alpha }z^{-1}1_{z\geq 1}}{\Gamma
(1-\alpha )\Gamma (\alpha )}=\frac{\sin (\pi \alpha )}{\pi }\frac{1_{z\geq 1}%
}{(z-1)^{\alpha }z},\qquad \alpha \in (0,1).
\end{equation*}%
For $\alpha =1$, the jumps are unitary and the process coincides with the
standard Poisson. The representation (\ref{cp}) can be checked directly as
follows: the Laplace transform of the addends $Z_{j}^{\alpha }$ is given by%
\begin{equation}
\mathbb{E}e^{-\eta Z_{j}^{\alpha }}=\frac{\sin (\pi \alpha )}{\pi }%
\int_{1}^{+\infty }(z-1)^{-\alpha }z^{-1}e^{-\eta z}dz=\frac{\Gamma (\alpha
;\eta )}{\Gamma (\alpha )},  \label{pr}
\end{equation}%
for any $j=1,2,...,$ by formula (3.383.9) in \cite{GRA} for $\alpha <1.$
Then, by conditioning, we get%
\begin{eqnarray*}
\mathbb{E}e^{-\eta \sum_{j=1}^{N_{\alpha }(t)}Z_{j}^{\alpha }} &=&\exp
\left\{ -\alpha \Gamma (\alpha )t\left[ 1-\frac{\Gamma (\alpha ;\eta )}{%
\Gamma (\alpha )}\right] \right\} \\
&=&\exp \left\{ -t\alpha \gamma (\alpha ;\eta )\right\} .
\end{eqnarray*}%
Finally, we note that $S_{\alpha }$ is not self-similar, as can be checked
from its Laplace transform.

The moments of any integer order of $S_{\alpha } $ are not finite, for any $%
t>0$, since
\begin{equation}
\int_{1}^{+\infty }|x|^{k}\pi (dx)=\int_{1}^{+\infty }\frac{\alpha
(x-1)^{-\alpha }x^{k-1}}{\Gamma (1-\alpha )}dx  \label{pp}
\end{equation}%
does not converge, for any $k\geq 1,$ (see \cite{APP} p.132). Alternatively,
this can be seen by applying the Wald formula and by noting that $\mathbb{E}%
Z_{j}^{\alpha }=\frac{\sin (\pi \alpha )}{\pi }\int_{1}^{+\infty
}(z-1)^{-\alpha }dz=+\infty $, $j=1,2,...,$

The reason can be found in the heaviness of its distribution's tail. Indeed
it can be proved that it displays the same power law of the stable
subordinator, i.e. $P(X_{\alpha }(t)>x)\simeq \frac{tx^{-\alpha }}{\Gamma
(1-\alpha)}$ for large $x$ (see \cite{SAM}, p.17).

However we can study the asymptotic expression of the fractional moment of $%
S_{\alpha },$ of order $p\leq \alpha $ and for large $t.$ We recall that the
fractional moments have been introduced and studied by many authors, in
order to overcome the problem of infinite integer order moments, especially
in the stable case (see, among the others, \cite{MAT} and \cite{WOL}); in
particular, we will follow the techniques given in \cite{KUM}, which are
based on fractional differentiation of the Laplace transform.

\begin{theorem}
1) Let $\alpha \in (0,1),$ then, for any $t\geq 0$ and for $x\rightarrow
+\infty ,$ we have that%
\begin{equation}
P(S_{\alpha }(t)>x)\simeq \frac{tx^{-\alpha }}{\Gamma (1-\alpha )}.
\label{pst}
\end{equation}%
2) Let $p\in (0,1]$, then the fractional moment of order $p$ of the process $%
S_{\alpha }$ exists, finite, for $p\leq \alpha $ and it asymptotically
behaves as follows%
\begin{equation}
\mathbb{E}S_{\alpha }^{p}(t)\simeq \frac{\Gamma \left( 1-\frac{p}{\alpha }%
\right) }{\Gamma \left( 1-p\right) }t^{p/\alpha },\qquad t\rightarrow
+\infty .  \label{cor}
\end{equation}
\end{theorem}

\begin{proof}
We can write, for $\eta \rightarrow 0,$%
\begin{eqnarray*}
\int_{0}^{+\infty }e^{-\eta x}P(S_{\alpha }(t)>x)dx &=&\frac{1-\mathbb{E}%
e^{-\eta S_{\alpha }(t)}}{\eta }=\frac{1-e^{-t \alpha \gamma (\alpha ;\eta )}%
}{\eta } \\
&\simeq & t \eta ^{\alpha -1},
\end{eqnarray*}%
where we have taken the Taylor series expansion (up to the first order) and
we have considered the asymptotic behavior of the lower incomplete gamma
function, i.e.%
\begin{equation}
\gamma (\alpha ;\eta )\simeq \frac{\eta ^{\alpha }}{\alpha },\qquad \eta
\rightarrow 0.  \label{sss}
\end{equation}%
Formula (\ref{sss}) can be easily derived by rewriting (\ref{incomplete
gamma function}) as follows:%
\begin{equation*}
\gamma (\alpha ;x)=x^{\alpha }\int_{0}^{1}e^{-xw}w^{\alpha -1}dw.
\end{equation*}%
By applying the Tauberian theorem (see \cite{FEL}, Thm.4, p.446) we get, for
any $t\geq 0$, result (\ref{pst}).

In order to derive the asymptotic behavior of the fractional moment of order
$p$, we apply the Laplace-Erdelyi Theorem to the following integral%
\begin{eqnarray*}
\mathbb{E}S_{\alpha }^{p}(t) &=&-\frac{1}{\Gamma \left( 1-p\right) }%
\int_{0}^{+\infty }\frac{d}{d\eta }\left[ e^{-t\alpha \gamma (\alpha ;\eta )}%
\right] \eta ^{-p}d\eta  \\
&=&\frac{\alpha t}{\Gamma \left( 1-p\right) }\int_{0}^{+\infty }e^{-\eta
-t\alpha \gamma (\alpha ;\eta )}\eta ^{\alpha -p-1}d\eta ,
\end{eqnarray*}%
(see \cite{WOI}, for details). Let $x\in (x_{0}$,$x_{1})$, with $x_{0}$,$%
x_{1}\in \mathbb{R}$, let moreover $h(x)$ and $\varphi (x)$\ be independent
of $t>0$ and $h(x)>h(x_{0})$ for all $x\in (x_{0},x_{1})$. Moreover, let the
following expansions hold, for $x\rightarrow x_{0}^{+},$ $h(x)\sim
h(x_{0})+\sum_{k=0}^{\infty }a_{k}(x-x_{0})^{k+\mu },$ $\mu \in \mathbb{R}%
^{+},$ $a_{0}\neq 0$, and $\varphi (x)\sim \sum_{k=0}^{\infty
}b_{k}(x-x_{0})^{k+\gamma -1},$ $\gamma \in \mathbb{R}^{+},$ $b_{0}\neq 0.$
Then
\begin{equation}
I(t):=\int_{x_{0}}^{x_{1}}\varphi (x)e^{-th(x)}dx\sim
e^{-th(x_{0})}\sum_{j=0}^{\infty }\frac{c_{j}}{t^{\frac{\gamma +j}{\mu }}}%
\Gamma \left( \frac{\gamma +j}{\mu }\right) ,\qquad t\rightarrow +\infty ,
\label{LE}
\end{equation}%
under the assumption that the integral (with finite or infinite delimiters)
converges absolutely for all sufficiently large $t$. We only need $%
c_{0}=b_{0}/\mu a_{0}^{\gamma /\mu }$, then, for the expressions of the
other $c_{j}$'s we refer to \cite{WOI} and \cite{WON}. In our case, we have
that $\varphi (x):=e^{-x}x^{\alpha -p-1}=\sum_{k=0}^{\infty }\frac{%
(-1)^{k}x^{k+\alpha -p-1}}{k!}$, so that $\gamma =\alpha -p>0$, for $%
p<\alpha $, and $b_{0}=1.$ On the other hand, we get $h(x):=\alpha \gamma
(\alpha ;x)=\alpha \gamma (\alpha ;0)+\alpha \sum_{k=0}^{\infty }\frac{%
(-1)^{k}x^{k+\alpha }}{k!(\alpha +k)},$ by using the well-known series
expression of the incomplete gamma function (see \cite{JAM}). Thus we get $%
\mu =\alpha $ and $a_{0}=1.$ By considering (\ref{LE}) we thus get%
\begin{equation*}
\mathbb{E}S_{\alpha }^{p}(t)\sim \frac{\alpha t}{\Gamma \left( 1-p\right) }%
\sum_{j=0}^{\infty }\frac{c_{j}}{t^{\frac{\alpha -p+j}{\alpha }}}\Gamma
\left( \frac{\alpha -p+j}{\alpha }\right) \sim \frac{\alpha c_{0}t^{p/\alpha
}\Gamma \left( 1-\frac{p}{\alpha }\right) }{\Gamma \left( 1-p\right) },
\end{equation*}%
which coincides with (\ref{cor}).
\end{proof}

\begin{remark}
The fractional moment of order $p$ converges, for $t\rightarrow +\infty $,
to the value obtained in the stable case, for any $t$ (see \cite{MATS}).
\end{remark}

\subsection{Link to stable subordinators}

\paragraph{The one-dimensional case}

\label{approssimazione subordinatore stabile}

We now purpose a slight generalization of the previous results, in order to
provide an approximation of a stable subordinator: while the previously
defined subordinator $S_{\alpha }$ performs jumps greater than $1$, we now
consider a lower bound for the jump size equal to $\epsilon >0$.

We thus define the following L\'{e}vy measure with support on $(\epsilon
,\infty )$ and with density
\begin{equation*}
\pi _{\epsilon }(x)=\frac{\alpha }{\Gamma (1-\alpha )}(x-\epsilon )^{-\alpha
}x^{-1}1_{x\geq \epsilon }.
\end{equation*}%
The corresponding Laplace exponent has the form
\begin{equation*}
\varphi _{\varepsilon }(\eta )=\frac{\alpha }{\epsilon ^{\alpha }}\,\gamma
(\alpha ;\eta \epsilon ).
\end{equation*}%
Indeed,
\begin{align}
\varphi _{\epsilon }(\eta )& =\int_{0}^{\infty }(1-e^{-\eta x})\pi
_{\epsilon }(x)dx  \notag \\
& =\int_{\epsilon }^{\infty }(1-e^{-\eta x})\frac{\alpha }{\Gamma (1-\alpha )%
}(x-\epsilon )^{-\alpha }x^{-1}dx  \notag \\
& =\int_{\epsilon }^{\infty }dx\frac{\alpha }{\Gamma (1-\alpha )}(x-\epsilon
)^{-\alpha }x^{-1}\int_{0}^{\eta }xe^{-xz}dz  \notag \\
& =\int_{0}^{\eta }dz\int_{0}^{\infty }\frac{\alpha y^{-\alpha }}{\Gamma
(1-\alpha )}e^{-z(y+\epsilon )}\,dy  \notag \\
& =\frac{\alpha }{\epsilon ^{\alpha }}\int_{0}^{\eta \epsilon
}e^{-w}w^{\alpha -1}dw.  \label{calcoli esponente di laplace}
\end{align}%
By a simple change of variable, the Laplace exponent can also be expressed
as $\eta ^{\alpha }$ multiplied by a correction factor depending on $\epsilon $:
\begin{equation*}
\varphi _{\epsilon }(\eta )=\eta ^{\alpha }\cdot O_{\epsilon }(\eta )
\end{equation*}%
where
\begin{equation}
O_{\epsilon }(\eta )=\frac{\alpha }{\epsilon ^{\alpha }}\int_{0}^{\epsilon
}e^{-\eta y}y^{\alpha -1}dy  \label{o piccolo}
\end{equation}%
is such that $O_{\epsilon }(\eta )\rightarrow 1,$ as $\epsilon \rightarrow 0$%
. Thus, in the limit as $\epsilon \rightarrow 0$, the related subordinator $%
S_{\alpha }^{(\varepsilon )}:=\left\{ S_{\alpha }^{(\varepsilon )}(t),t\geq
0\right\} $ converges to a $\alpha $-stable subordinator, since
\begin{align*}
& \pi _{\epsilon }(x)\rightarrow \frac{\alpha }{\Gamma (1-\alpha )}%
x^{-\alpha -1}1_{x\geq 0} \\
& \varphi _{\epsilon }(\eta )\rightarrow \eta ^{\alpha }.
\end{align*}%
By considering that
\begin{equation*}
\int_{0}^{\infty }\pi _{\epsilon }(x)dx=\alpha \Gamma (\alpha )\epsilon ^{-\alpha },
\end{equation*}%
we can conclude that $S_{\alpha }^{(\varepsilon )}$ is a compound Poisson
process, i.e.
\begin{equation*}
S_{\alpha }^{(\varepsilon )}(t)=\sum_{j=1}^{N^{\epsilon }(t)}Z_{j}^{\epsilon
}
\end{equation*}%
where $N^{\epsilon }(t)$ is a Poisson process with intensity $\alpha \Gamma
(\alpha )\epsilon ^{-\alpha }$ and $Z_{j}^{\epsilon }$ has density
\begin{equation*}
f_{Z_{j}^{\varepsilon }}(z)=\frac{\epsilon ^{\alpha }}{\Gamma (\alpha
)\Gamma (1-\alpha )}(z-\epsilon )^{-\alpha }z^{-1}1_{z\geq \epsilon }.
\end{equation*}%
Thus $S_{\alpha }^{(\varepsilon )}$ is a compound Poisson approximation of a
stable subordinator. Therefore, it can be useful in many applications,
since it is easier to be handled with respect to the stable subordinator,
thanks to its finite activity.

As far as the governing equation is concerned, we can show that the
transition density of $S_{\alpha }^{(\varepsilon )}$ satisfies a fractional
equation, which generalizes the governing equation (\ref{riemann}) of the
stable subordinator. In particular, the fractional derivative on the right
side is corrected by means of the following operator
\begin{align}
O_{\epsilon }\bigl (\frac{\partial }{\partial x}\bigr)h(x)& :=\frac{\alpha }{%
\epsilon ^{\alpha }}\int_{0}^{\epsilon }e^{-y\partial _{x}}h(x)y^{\alpha
-1}dy  \notag \\
& =\frac{\alpha }{\epsilon ^{\alpha }}\int_{0}^{\epsilon }h(x-y)y^{\alpha
-1}dy,  \label{op}
\end{align}%
where $h:\mathbb{R}^{+}\rightarrow \mathbb{R}$ is a function such that the
above integral converges, while $e^{-y\partial _{x}}$ denotes (with a little
abuse of notation), the translation operator.

Note that (\ref{op}) tends to the identity operator as $\epsilon \rightarrow
0$, since%
\begin{equation*}
\lim_{\epsilon \rightarrow 0}O_{\epsilon }\bigl (\frac{\partial }{\partial x}%
\bigr)h(x)=h(x).
\end{equation*}%
Thus we can check that the density $q_{\varepsilon }:=q_{\varepsilon }(x,t),$
$x,t\geq 0,$ of $S_{\alpha }^{(\varepsilon )}$ solves the following equation
\begin{equation*}
\frac{\partial }{\partial t}q_{\epsilon }(x,t)=-\frac{\partial ^{\alpha }}{%
\partial x^{\alpha }}O_{\epsilon }\bigl (\frac{\partial }{\partial x}\bigr)%
q_{\epsilon }(x,t)\qquad q_{\epsilon }(x,0)=\delta (x),
\end{equation*}%
by applying the Laplace transform to both members, which gives
\begin{equation*}
\widetilde{q}_{\epsilon }(\eta ,t)=e^{-\eta ^{\alpha }O_{\epsilon }(\eta )t},
\end{equation*}%
where $O_{\epsilon }(\eta )$ has been defined in (\ref{o piccolo}).

\begin{remark}
The approximation presented above could be applied to the fractional derivative with
time-dependent order, i.e. $\bigl( \frac{\partial}{\partial x}\bigr )^{\alpha (t)} $,
where $\alpha (t)$ takes values in $(0,1)$. Such operator governs a time-inhomogeneous
version of the stable subordinator (see, for example, \cite{BEG3} and \cite{RIC}), which
could be approximated by considering the time-dependent L\'{e}vy measure
$\pi _{\epsilon }(x,t)=\frac{\alpha (t) }{\Gamma (1-\alpha (t) )}(x-\epsilon )^{-\alpha (t)
}x^{-1}1_{x\geq \epsilon }.$
\end{remark}

\paragraph{The multivariate case}

Following the lines of the one-dimensional case, we look for a compound
Poisson approximation for a multivariate stable subordinator, which we
introduced in sect. \ref{sottosezione multivariate}. We define the family of
L\'{e}vy measures
\begin{equation*}
\nu _{\epsilon }(d\rho ,d\theta )=C(\rho -\epsilon )^{-\alpha }\rho
^{-1}M(d\theta )\qquad \epsilon >0
\end{equation*}%
and, by the same calculations as in (\ref{calcoli esponente di laplace}), we
obtain the following family of Bernstein functions (the symbol $\eta $
denotes the vector $(\eta _{1},\dots ,\eta _{d})$ and $\cdot $ denotes the
scalar product)
\begin{align*}
\Phi _{\epsilon }(\eta )& =k\int_{0}^{\infty }d\rho \int_{B_{+}^{d-1}}\bigl (%
1-e^{-\rho \,\eta \cdot \theta })(\rho -\epsilon )^{-\alpha }\rho
^{-1}M(d\theta ) \\
& =k\int_{B_{+}^{d-1}}(\theta \cdot \eta )^{\alpha }\,O_{\epsilon }(\eta
\cdot \theta )M(d\theta )
\end{align*}%
where the corrective term
\begin{equation*}
O_{\epsilon }(\eta \cdot \theta )=\frac{\alpha }{\epsilon ^{\alpha }}%
\int_{0}^{\epsilon }e^{-\eta \cdot \theta y}y^{\alpha -1}dy
\end{equation*}%
tends to $1$ as $\epsilon \rightarrow 0$. By Laplace inversion, the density $%
q_{\epsilon }(x,t)$ of our process satisfies
\begin{equation*}
\frac{\partial }{\partial t}q_{\epsilon }(x,t)=-k\int_{B_{+}^{d-1}}(\nabla
_{x}\cdot \theta )^{\alpha }O_{\epsilon }(\theta \cdot \nabla
_{x})q_{\epsilon }(x,t)\,M(d\theta )
\end{equation*}%
where
\begin{align*}
O_{\epsilon }(\theta \cdot \nabla _{x})h(x):= \frac{\alpha }{\epsilon
^{\alpha }}\int_{0}^{\epsilon } e^{-y \theta \cdot \nabla _x} h(x)y^{\alpha -1}dy
\\
=\frac{\alpha }{\epsilon ^{\alpha }}\int_{0}^{\epsilon }h(x-y\theta
)y^{\alpha -1}dy
\end{align*}%
tends to the identity operator in the limit $\epsilon \rightarrow 0$.

\section{The tempered subordinator $S_{\protect\alpha ,\protect\theta }$}

In order to avoid the inconvenience of infinite moments of $S_{\alpha }$, we
define a tempered counterpart of the latter.

\begin{theorem}
Let $\eta ,\theta >0$ and $\alpha \in (0,1],$ then the function
\begin{equation}
\varphi _{\theta }(\eta ):=\alpha \gamma (\alpha ;\eta +\theta )-\alpha
\gamma (\alpha ;\theta ),  \label{tm}
\end{equation}%
is the Laplace exponent of a tempered subordinator $S_{\alpha ,\theta
}:=\left\{ S_{\alpha ,\theta }(t),t\geq 0\right\} $, with L\'{e}vy triplet $%
(0,0,\pi _{\theta })$ and (absolutely continuous) L\'{e}vy measure $\pi
_{\theta }$, with density%
\begin{equation}
\overline{\pi }_{\theta }(z)=\frac{1_{z\geq 1}\alpha (z-1)^{-\alpha
}z^{-1}e^{-\theta z}}{\Gamma (1-\alpha )}.  \label{tm2}
\end{equation}%
The sample paths of $S_{\alpha ,\theta }$ are not strictly increasing; the
mean and variance of $S_{\alpha ,\theta }$ read, respectively,%
\begin{align}
\mathbb{E}S_{\alpha ,\theta }(t)& =t\,\alpha \theta ^{\alpha -1}e^{-\theta }
\label{mv} \\
\mathbb{V}arS_{\alpha ,\theta }(t)& =t\,\alpha \theta ^{\alpha -1}e^{-\theta
}+\alpha (1-\alpha )t\theta ^{\alpha -2}e^{-\theta }.  \notag
\end{align}
\end{theorem}

\begin{proof}
It is immediate to check that (\ref{tm}) is a Bernstein function (as a
consequence of Theorem 1). Moreover we can prove that the representation (%
\ref{lk}) holds, in this case, for $a=b=0$ and for the L\'{e}vy density
given in (\ref{tm2}): indeed we have that%
\begin{eqnarray*}
\int_{0}^{+\infty }(1-e^{-\eta x})\overline{\pi }_{\theta }(x)dx
&=&\int_{0}^{+\infty }x\int_{0}^{\eta }e^{-zx}dz\overline{\pi }_{\theta
}(x)dx \\
&=&\int_{0}^{\eta }dz\int_{1}^{+\infty }xe^{-zx}\frac{\alpha (x-1)^{-\alpha
}x^{-1}e^{-\theta x}}{\Gamma (1-\alpha )}dx \\
&=&\int_{0}^{\eta }e^{-(z+\theta )}dz\int_{0}^{+\infty }e^{-(z+\theta )w}%
\frac{\alpha w^{-\alpha }}{\Gamma (1-\alpha )}dw \\
&=&\alpha \int_{\theta }^{\theta +\eta }e^{-w}w^{\alpha -1}dz,
\end{eqnarray*}%
which coincides with (\ref{tm}). Also in this case the L\'{e}vy measure is
finite, since%
\begin{equation}
\int_{0}^{+\infty }\pi _{\theta }(dz)=\int_{1}^{+\infty }\frac{\alpha
(z-1)^{-\alpha }z^{-1}e^{-\theta z}}{\Gamma (1-\alpha )}dz=\alpha \Gamma
(\alpha ;\theta )<\infty ,  \label{mm}
\end{equation}%
by considering (\ref{pr}). The mean and variance given in (\ref{mv}) can be
obtained by differentiating the Laplace transform
\begin{equation}
\mathbb{E}e^{-\eta S_{\alpha ,\theta }(t)}=e^{-t\alpha \int_{\theta
}^{\theta +\eta }e^{-w}w^{\alpha -1}dz},  \label{cc}
\end{equation}%
with respect to $\eta $ and considering the relationship $\mathbb{E}\left[
S_{\alpha ,\theta }(t)\right] ^{k}=(-1)^{k}\left. \frac{\partial ^{k}}{%
\partial \eta ^{k}}\mathbb{E}e^{-\eta S_{\alpha ,\theta }(t)}\right\vert
_{\eta =0},$ for $k\in \mathbb{N}.$
\end{proof}

\begin{remark}
It is easy to check that the mean and variance of $S_{\alpha ,\theta }$,
given in (\ref{mv}), tend to infinity, as $\theta \rightarrow 0,$ as
expected from (\ref{pp}).

\end{remark}

\begin{remark}
From (\ref{mm}) we can infer that the process $S_{\alpha ,\theta }$ has
finite activity and the following compound Poisson representation holds
\begin{equation}
S_{\alpha ,\theta }(t)=\sum_{j=1}^{N_{\alpha ,\theta }(t)}Z_{j}^{\alpha
,\theta },  \label{cp2}
\end{equation}%
where $N_{\alpha ,\theta }:=\left\{ N_{\alpha ,\theta }(t),t\geq 0\right\} $
is a homogeneous Poisson process with rate $\lambda =\alpha \Gamma (\alpha
;\theta ).$ The jumps $Z_{j}^{\alpha ,\theta }$ are i.i.d. random variables,
taking values in $[1,+\infty )$ and with probability density function
\begin{equation*}
f_{Z^{\alpha ,\theta }}(z)=\frac{1_{z\geq 1}(z-1)^{-\alpha }z^{-1}e^{-\theta
z}}{\Gamma (1-\alpha )\Gamma (\alpha ;\theta )},\qquad \alpha \in (0,1).
\end{equation*}%
For $\alpha =1,$ formula (\ref{tm}) reduces to $\varphi _{\theta }(\eta
)=\gamma (1;\eta +\theta )-\gamma (1;\theta )=e^{-\vartheta }(e^{-\eta }-1)$%
, which is the Laplace exponent of a Poisson process of rate $e^{-\theta }.$
This is confirmed by its L\'{e}vy measure, which is obtained from (\ref{tm2}%
), since $\lim_{\alpha \rightarrow 1}\pi _{\theta }(dz)=e^{-\theta }\delta
_{1}(z)dz$. Indeed, the process $N_{1,\theta }$ in (\ref{cp2}) has rate $%
\lambda =\Gamma (1;\theta )=e^{-\vartheta },$ in this special case.
\end{remark}

\section{Subordination of L\'evy processes}

We now consider the subordination of a L\'{e}vy process $X(t)$ by means of $%
\beta _0 t +S_{\alpha ,\theta }(t)$, where $S_{\alpha ,\theta }$ is the
tempered subordinator defined above and $\beta _0\geq 0$ is a possible drift
parameter. Let $(a,b,\nu )$ be the L\'evy triplet of $X$ and $\mu$ be its
probability distribution, i.e. $\mu _{t}\left( B\right) :=P\left( X(t)\in
B\right),$ for any Borel set $B$. We assume that $X$ is independent of $%
S_{\alpha,\theta }$.

Then, by applying Thm. 30.1, p.197 in \cite{SAT}, the process $Z:=\left\{
Z(t),t\geq 0\right\} $ defined as
\begin{equation}
Z(t):=X(\beta _{0}t+S_{\alpha ,\theta }(t)),\qquad t\geq 0,\text{ }
\label{zeta}
\end{equation}%
is a L\'{e}vy process with triplet $(a^{\prime },b^{\prime },\nu ^{\prime }),
$ where
\begin{align}
a^{\prime }& =\beta _{0}a  \label{z2} \\
b^{\prime }& =\beta _{0}b+\int_{0}^{+\infty }\pi _{\theta }(dz)\int_{|x|\leq
1}x\mu _{z}(dx)  \notag \\
\nu ^{\prime }(dx)& =\beta _{0}\nu (dx)+\int_{1}^{+\infty }\mu _{z}(dx)\pi
_{\theta }(dz)  \notag
\end{align}%
By considering Prop.1.3.27 in \cite{APP}, we can also derive the L\'{e}vy
symbol of the subordinated process, which is again expressed in terms of
incomplete gamma functions:%
\begin{equation}
\psi _{Z}(u)=-\varphi _{\theta }(-\psi _{X}(u))=\alpha \gamma (\alpha
;\theta )-\alpha \gamma (\alpha ;\theta -\psi _{X}(u)).  \label{z3}
\end{equation}

\subsection{The generator equation.}

Let us consider the case $\beta _{0}=\theta =0$. For $h\in B_{b}(\mathbb{R})$%
, where $B_{b}(\mathbb{R})$ denotes the set of real-valued bounded Borel
measurable functions, equipped with the sup-norm, the operator $T_{t}$
defined by
\begin{equation}
T_{t}\,h(x)=\mathbb{E}\,h\bigl (x+X(t)\bigr ) \label{semigruppo}
\end{equation}%
defines a strongly continuous contraction semigroup on $B_{b}(\mathbb{R})$.
If $A$ is the generator of $T_{t}$, then (\ref{semigruppo}) satisfies
\begin{equation*}
\frac{\partial }{\partial t}g(x,t)=Ag(x,t)\qquad g(x,0)=h(x)
\end{equation*}%
for $h$ in the domain of $A$. If $\sigma _{\alpha }(t)$ is a stable
subordinator, then the process $X(\sigma _{\alpha }(t))$ induces the
subordinate semigroup
\begin{equation}
\tilde{T}_{t}\,h(x)=\mathbb{E}\,h\bigl (x+X(\sigma _{\alpha }(t))\bigr ).
\label{ph}
\end{equation}%
In light of the Phillips theorem (see \cite{SAT} page 212), the semigroup (%
\ref{ph}) satisfies
\begin{equation}
\frac{\partial }{\partial t}g(x,t)=-(-A)^{\alpha }g(x,t),\qquad g(x,0)=h(x),
\label{equazione potenza di operatore}
\end{equation}
where the fractional power of the operator is defined by
\begin{equation}
-(-A)^{\alpha }h(x)=\int_{0}^{\infty }\bigl (T_{s}h(x)-h(x)\bigr )\frac{%
\alpha }{\Gamma (1-\alpha )}s^{-\alpha -1}ds  \label{fractional power}
\end{equation}%
at least on the same domain of $A$.

Now, if we employ the subordinator $S_{\alpha }^{(\epsilon )}$, which is an
approximation of $\sigma _{\alpha }$ (see the discussion in section \ref%
{approssimazione subordinatore stabile}), we obtain an approximation of
equation (\ref{equazione potenza di operatore}). Indeed, using again the
Phillips theorem,
\begin{equation*}
T_{t}^{\epsilon }h(x)=\mathbb{E}\,h(x+X(S_{\alpha }^{\epsilon }(t))
\end{equation*}%
satisfies the following equation
\begin{equation*}
\frac{\partial }{\partial t}g(x,t)=\int_{\epsilon }^{\infty
}(T_{s}g(x,t)-g(x,t))\frac{\alpha (s-\epsilon )^{-\alpha }s^{-1}}{\Gamma
(1-\alpha )}\,ds,\qquad g(x,0)=h(x).
\end{equation*}%
The operator on the right-side is an approximation of the fractional power
in (\ref{fractional power}), to which it converges as $\epsilon \rightarrow 0
$.

We observe that, in the special case $X(t)=t$, i.e. when $T_{t}$ is the
shift operator, the operator on the right-side is an approximation of the
Marchaud fractional derivative, namely
\begin{equation*}
\int_{\epsilon }^{\infty }(g(x-s,t)-g(x,t))\frac{\alpha (s-\epsilon
)^{-\alpha }s^{-1}}{\Gamma (1-\alpha )}\,ds.
\end{equation*}

\subsection{Subordinated Brownian motion}

In the Brownian case, i.e. when the external process $X$ is a standard
Brownian motion $B:=\left\{ B(t),t\geq 0\right\} $ and the triplet is $%
(0,1,0)$, we have, from (\ref{z2}), that the L\'{e}vy process
\begin{equation*}
Z(t)=B(\beta _{0}t+S_{\alpha ,\theta }(t))
\end{equation*}%
is given by the superposition of a Brownian motion (with diffusion
coefficient $\beta _{0}$) and a jump process. Indeed it has L\'{e}vy triplet
$(0,\beta _{0},\nu ^{\prime })$, where
\begin{align*}
\nu ^{\prime }(x)& =\int_{1}^{+\infty }\frac{e^{-\frac{x^{2}}{2z}}}{\sqrt{%
2\pi z}}\frac{\alpha (z-1)^{-\alpha }z^{-1}e^{-\theta z}}{\Gamma (1-\alpha )}%
dz \\
& =\frac{\alpha }{\Gamma (1-\alpha )\sqrt{2\pi }}\sum_{j=0}^{+\infty }\frac{%
(-x^{2}/2)^{j}}{j!}\int_{1}^{+\infty }(z-1)^{1-\alpha -1}z^{-j-\frac{1}{2}%
-1}e^{-\theta z}dz \\
& =[\text{by (3.383.4) in \cite{GRA}}] \\
& =\frac{\alpha e^{-\theta /2}\theta ^{\frac{\alpha }{2}-\frac{1}{4}}}{\sqrt{%
2\pi }}\sum_{j=0}^{+\infty }\frac{(-x^{2}\sqrt{\theta }/2)^{j}}{j!}W_{\frac{%
\alpha }{2}-\frac{j}{2}-\frac{3}{4},\frac{\alpha }{2}+\frac{j}{2}+\frac{1}{4}%
}(\theta (1-\alpha )),
\end{align*}%
where $W_{\beta ,\gamma }(\cdot )$ denotes the Whittaker function (see also
\cite{MAT}, p.27), by considering that $1-\alpha >0$ and $\theta >0.$ In the
special case where $\theta =0,$ i.e. in the non-tempered case, we have
instead the following easier expression%
\begin{align}
\nu ^{\prime }(x)& =\int_{1}^{+\infty }\frac{e^{-\frac{x^{2}}{2z}}}{\sqrt{%
2\pi z}}\frac{\alpha (z-1)^{-\alpha }z^{-1}}{\Gamma (1-\alpha )}dz
\label{mm2} \\
& =\int_{0}^{1}\frac{e^{-\frac{x^{2}w}{2}}}{\sqrt{2\pi }}\frac{\alpha
w^{\alpha +\frac{1}{2}-1}(1-w)^{1-\alpha -1}}{\Gamma (1-\alpha )}dw  \notag
\\
& =[\text{by (1.6.15) in \cite{KIL} }a=\alpha +\frac{1}{2},\text{ }c=\frac{3%
}{2}]  \notag \\
& =\frac{\sqrt{2}\alpha \Gamma \left( \alpha +\frac{1}{2}\right) }{\pi }%
\,_{1}F_{1}\left( \alpha +\frac{1}{2};\frac{3}{2};-\frac{x^{2}}{2}\right) ,
\notag
\end{align}%
where $_{1}F_{1}\left( a;c;z\right) =\sum_{k=0}^{\infty }\frac{\left(
a\right) _{k}}{(c)_{k}}\frac{z^{k}}{k!}$ is the confluent hypergeometric
Kummer function, which is defined for any $a,z\in \mathbb{C}$ and $c\in
\mathbb{C}\backslash \mathbb{Z}_{0}^{-}$ (see \cite{KIL}, p.29, for
details). Thanks to formula (1.9.3) in \cite{KIL}, p.45, we can write (\ref%
{mm2}) in terms of the generalized (three-parameters) Mittag-Leffler
function, as follows%
\begin{equation}
\nu ^{\prime }(x)=\alpha \frac{\Gamma \left( \alpha +\frac{1}{2}\right) }{%
\sqrt{2\pi }}E_{1,3/2}^{\alpha +1/2}\left( -\frac{x^{2}}{2}\right) ,  \notag
\end{equation}%
where $E_{\alpha ,\beta }^{\gamma }\left( z\right) :=\sum_{k=0}^{\infty }%
\frac{(\gamma )_{k}z^{k}}{k!\Gamma (\alpha k+\beta )}$ and $(\gamma
)_{k}:=\gamma (\gamma +1)...(\gamma +n-1)$, for $z,\alpha ,\beta ,\gamma \in
\mathbb{C}$ with $Re(\alpha )>0$, $n\in \mathbb{N}.$

It is easy to check that the jump component of the subordinated process has
finite activity for any $\alpha \in (0,1)$, since%
\begin{eqnarray*}
\int_{0}^{+\infty }\nu ^{\prime }(dx) &=&\int_{1}^{+\infty }\frac{\alpha
(z-1)^{-\alpha }z^{-1}e^{-\theta z}}{\Gamma (1-\alpha )}dz \\
&=&[\text{by (\ref{mm})}] \\
&=&\alpha \Gamma (\alpha ;\theta )<\infty .
\end{eqnarray*}%
Moreover we have that%
\begin{eqnarray*}
\int_{|x|\geq 1}|x|^{k}\nu ^{\prime }(dx) &=&\int_{1}^{+\infty }\left(
\int_{|x|\geq 1}|x|^{k}\frac{e^{-\frac{x^{2}}{2z}}}{\sqrt{2\pi z}}dx\right)
\frac{\alpha (z-1)^{-\alpha }z^{-1}e^{-\theta z}}{\Gamma (1-\alpha )}dz \\
&\leq &\int_{1}^{+\infty }\mathbb{E}|B(z)|^{k}\frac{\alpha (z-1)^{-\alpha
}z^{-1}e^{-\theta z}}{\Gamma (1-\alpha )}dz \\
&=&\frac{2^{k/2}\Gamma \left( \frac{k+1}{2}\right) }{\sqrt{\pi }}%
\int_{1}^{+\infty }\frac{\alpha (z-1)^{1-\alpha -1}z^{\frac{k}{2}%
-1}e^{-\theta z}}{\Gamma (1-\alpha )}dz
\end{eqnarray*}

The characteristic function of $Z(t)$ is given by
\begin{align}
\mathbb{E}e^{iuB(\beta _0t+S_{\alpha ,\theta }(t))}=\exp \left\{-\frac{1}{2}%
u^2 \beta _0t -t\, \alpha \int_{\theta }^{\theta +u^{2}/2}e^{-w}w^{\alpha
-1}dw\right\}.
\end{align}

By conditioning and considering (\ref{mv}), we have that $\mathbb{E}B(\beta
_{0}t+S_{\alpha ,\theta }(t))=0,$ for any $t,\theta \geq 0,$ and the
autocovariance of the subordinated Brownian motion, for any $t,\tau \geq 0,$
reads%
\begin{align}
Cov\left( B(\beta _{0}t+S_{\alpha ,\theta }(t)),B(\beta _{0}\tau +S_{\alpha
,\theta }(\tau ))\right) & =\mathbb{E}((\beta _{0}t+S_{\alpha ,\theta
}(t))\wedge (\beta _{0}\tau +S_{\alpha ,\theta }(\tau ))  \notag \\
& =\mathbb{E}(\beta _{0}(t\wedge \tau )+S_{\alpha ,\theta }(t\wedge \tau ))
\notag \\
& =\beta _{0}(t\wedge \tau )+(t\wedge \tau )\alpha \theta ^{\alpha
-1}e^{-\theta }.  \label{z4}
\end{align}%
Thus, even if the autocovariance is linear w.r.t. the time argument, the
parameters $\alpha $ and $\theta $ can be interpreted as a measure of
deviation from the standard Brownian dependence structure: in particular,
for $\theta \rightarrow 0$ and for $\alpha $ strictly less than $1$, the
autocovariance tends to infinity, for any $t.$

\section{Subordinated fractional Brownian motion}

We now consider the process $\left\{ B_{H}(S_{\alpha }(t)),t\geq 0\right\} $%
, where $B_{H}:=\left\{ B_{H}(t),t\geq 0\right\} $ is the fractional
Brownian motion (hereafter FBM) with Hurst parameter $H$ and the
subordinator $S_{\alpha }$ is supposed to be independent of it. The FBM $%
B_{H}$ is defined, for any $H\in (0,1)$ as a self-similar process with index
$H$ and with zero-mean Gaussian distribution. Its one dimensional
distribution has density
\begin{equation*}
f_{B_{H}}(x,t)=\frac{1}{\sqrt{2\pi }t^{H}}\exp \left\{ -\frac{x^{2}}{2t^{2H}}%
\right\} ,\qquad x\in \mathbb{R},\text{ }t\geq 0.
\end{equation*}%
It can be expressed, in terms of the standard Brownian motion $B:=\left\{
B(t),t\geq 0\right\} ,$ by the following representation%
\begin{equation*}
B_{H}(t)=\int_{\mathbb{R}}\left[ (t-u)_{+}^{H-1/2}-(-u)_{+}^{H-1/2}\right]
dB(u),\qquad t\geq 0
\end{equation*}%
where $x_+= \max(x,0)$.

For details on the fractional Brownian motion we refer to \cite{MAN}. It is
worth recalling that the FBM exhibits subdiffusive dynamics for $H<1/2$ and
a superdiffusive one for $H>1/2$; indeed the moment of order $q$ of the FBM
is given by
\begin{equation}
\mathbb{E}\left\vert B_{H}\left( t\right) \right\vert ^{q}=t^{qH}\mathbb{E}%
\left\vert B_{H}(1)\right\vert ^{q}=\sqrt{\frac{2^{q}}{\pi }}\Gamma \left(
\frac{q+1}{2}\right) t^{qH}.  \label{ku}
\end{equation}%
(see, for example, \cite{KUM}).

Different forms of time-changed FBM have been introduced and studied (see
\cite{KUM}, \cite{KUM2}, \cite{MIJ}).

We prove here that the FBM, subordinated by an independent $S_{\alpha },$
displays the long-range dependence (LRD) property, for $H\in (0,1/2)$;
moreover, this behavior depends on $\alpha $, instead of what happens in the
cases of the FBM subordinated by the tempered stable subordinator (studied
in \cite{KUM}) and by the gamma process (analyzed in \cite{KUM2}). Indeed,
in the last cases, the LRD rate depends only on the Hurst parameter $H.$

Since the process is not stationary, we use the following definition of
long-range dependence: a process $Z(t)$ is said to have LRD property if, for
$s>0$ and $t>s,$.%
\begin{equation}
Corr(Z(t),Z(s))\sim c(s)t^{-d},\qquad t\rightarrow +\infty ,  \label{f2}
\end{equation}%
where $c(s)$ is a constant depending on $s$ and $d\in (0,1)$ (see \cite{DOV}%
).

\begin{theorem}
Let $H\in (0,1/2)$ and $\alpha \geq 2H$. Let%
\begin{equation}
Z_{H}(t):=B_{H}(S_{\alpha }(t)),\qquad t\geq 0,\text{ }
\end{equation}%
where $B_{H}$ is the FBM, with Hurst parameter $H$ and $S_{\alpha }$ is
supposed to be independent of it. Then $Z_{H}$ has the LRD behavior given in
(\ref{f2}), with $d=1-\frac{H}{\alpha }.$
\end{theorem}

\begin{proof}
We notice that the subordinator, being a compound Poisson process has
stationary and independent increments. By conditioning and considering (\ref%
{ku}), we get, for $q<\alpha /H,$
\begin{eqnarray}
\mathbb{E}|Z_{H}(t)|^{q} &=&\mathbb{E}|B_{H}(1)|^{q}\mathbb{E}\left(
S_{\alpha }(t)\right) ^{qH}=\sqrt{\frac{2^{q}}{\pi }}\Gamma \left( \frac{q+1%
}{2}\right) \mathbb{E}\left( S_{\alpha }(t)\right) ^{qH}  \label{h2} \\
&=&[\text{by (\ref{cor})}]  \notag \\
&\simeq &\sqrt{\frac{2^{q}}{\pi }}\Gamma \left( \frac{q+1}{2}\right) \frac{%
\Gamma \left( 1-\frac{qH}{\alpha }\right) }{\Gamma \left( 1-qH\right) }%
t^{qH/\alpha },\qquad t\rightarrow +\infty .  \notag
\end{eqnarray}%
We thus evaluate the covariance of the process $Z_{H},$ as follows , for $s<t
$,

\begin{eqnarray*}
\mathbb{E}\left( Z_{H}(t),Z_{H}(s)\right)  &=&\frac{1}{2}\left\{ \mathbb{E(}%
Z_{H}(t))^{2}+\mathbb{E(}Z_{H}(s))^{2}-\mathbb{E}\left[ Z_{H}(t)-Z_{H}(s)%
\right] ^{2}\right\}  \\
&=&\frac{1}{2}\mathbb{E}\left( B_{H}(1)\right) ^{2}\left\{ \mathbb{E}\left(
S_{\alpha }(t)\right) ^{2H}+\mathbb{E}\left( S_{\alpha }(s)\right) ^{2H}-%
\mathbb{E(}S_{\alpha }(t-s))^{2H}\right\}  \\
&=&[\text{by (\ref{cor})}] \\
&\sim &\frac{1}{2}\left\{ \frac{\Gamma \left( 1-\frac{2H}{\alpha }\right) }{%
\Gamma \left( 1-2H\right) }t^{2H/\alpha }+\mathbb{E}\left( S_{\alpha
}(s)\right) ^{2H}-\frac{\Gamma \left( 1-\frac{2H}{\alpha }\right) }{\Gamma
\left( 1-2H\right) }(t-s)^{2H/\alpha }\right\}  \\
&=&\frac{1}{2}\frac{\Gamma \left( 1-\frac{2H}{\alpha }\right) }{\Gamma
\left( 1-2H\right) }t^{2H/\alpha }\left\{ \frac{2H}{\alpha }\frac{s}{t}+%
\mathbb{E}\left( S_{\alpha }(s)\right) ^{2H}\frac{\Gamma \left( 1-2H\right)
}{\Gamma \left( 1-\frac{2H}{\alpha }\right) }t^{-2H/\alpha
}+O(t^{-2})\right\} .
\end{eqnarray*}%
By putting $K_{2H,\alpha }:=\Gamma \left( 1-\frac{2H}{\alpha }\right)
/\Gamma \left( 1-2H\right) $, we can write $\mathbb{E}\left(
Z_{H}(t),Z_{H}(s)\right) \sim \frac{H}{\alpha }K_{2H,\alpha }st^{\frac{2H}{%
\alpha }-1}$. Therefore, the correlation function asymptotically behaves as
follows, for $t\rightarrow +\infty ,$%
\begin{equation}
Corr(Z_{H}(t),Z_{H}(s))\sim \frac{st^{\frac{2H}{\alpha }-1}}{\sqrt{t^{\frac{%
2H}{\alpha }}s^{\frac{2H}{\alpha }}}}=s^{1-\frac{H}{\alpha }}t^{-(1-\frac{H}{%
\alpha })}.  \label{h}
\end{equation}%
Note that we have applied (\ref{h2}) for $q=2$ and thus (\ref{h}) holds for $%
\alpha \geq 2H,$ by Theorem 3; as a consequence, the result is limited to
the case of a FBM with $H<1/2.$
\end{proof}

\begin{remark}
We underline that the values of $H\geq 1/2$ are excluded, since, in this
range, the $\mathbb{E}\left( S_{\alpha }(t)^{2H}\right) $ is infinite. To
overcome this limitation we could have used the tempered subordinator $%
S_{\alpha ,\theta }(t)$ (as done in \cite{KUM}, in the stable case);
unfortunately, in the tempered case, the function $h(x)$ in (\ref{LE}) would
be given by $h(x)=\alpha \gamma (\alpha ;x+\theta )-\alpha \gamma (\alpha
;\theta )$, which cannot be expanded, as requested by the Laplace-Erdelyi
Theorem.
\end{remark}

\begin{remark}
We stress that the LRD parameter $d$ is dependent on $\alpha ,$
on the contrary of what happens in the case of a FBM subordinated by a
tempered stable subordinator or by the gamma process, where the rate $d$ of
the LRD depends only on the Hurst parameter $H$ and coincides with that of
the fractional Brownian motion itself (see \cite{KUM} and \cite{KUM2},
respectively).
\end{remark}

It is evident by (\ref{h2}) that $var(Z_{H}(t))\simeq Kt^{2H/\alpha },$ for $%
t\rightarrow +\infty $\ (where $K$ is a constant depending on $\alpha ,H$)
and therefore the process $Z_{H}$ behaves asymptotically as a subdiffusion,
according to the parameter $\alpha .$ Indeed, $2H/\alpha $ is always less
than one (since, by assumption, $2H\leq \alpha$) and the subdiffusive
behavior is more marked the greater the value of $\alpha $ (for any fixed $H$%
).


\begin{thebibliography}{99}
\bibitem{APP} Applebaum D., \emph{L\'{e}vy Processes and Stochastic
Calculus, }2nd edition, Cambridge University Press, Cambridge, 2009.

\bibitem{BEG} Beghin L., Gajda J., Tempered relaxation equation and related
generalized stable processes, \emph{Fractional Calculus and Applied Analysis}
23 (5), 2020, 1248, DOI: 10.1515/fca-2020-0063.

\bibitem{BEG2} Beghin L., Macci C., Ricciuti C., Random time-change with
inverses of multivariate subordinators: governing equations and fractional
dynamics, \emph{Stochastic Processes and their Applications}, 2020, 130
(10), 6364-6387.

\bibitem{BEG3} Beghin L., Ricciuti C., Additive geometric stable processes
and related pseudo-differential operators,
\emph{Markov Processes and Related Fields}, 25, 2019, 415-444.

\bibitem{BER} Bertoin J., Subordinators: Examples and Applications. Lectures
on probability theory and statistics (Saint-Flour, 1997), 1, 91. \emph{%
Lectures Notes in Math.}, 1717, Springer, Berlin, 1999.

\bibitem{garra} D'Ovidio M., Garra R., Multidimensional fractional
advection-dispersion equations and related stochastic processes. Electron.
J. Probab. 19 , no. 61, 2014.

\bibitem{DOV} D'Ovidio M., Nane E., Time dependent random fields on
spherical non-homogeneous surfaces, \emph{Stoch. Process. Appl., }124 (6),
2014, 2098-2131.

\bibitem{sato} Barndorff-Nielsen O., Pedersen J., Sato K., Multivariate
subordination, self-decomposability and stability. Adv. in Appl. Probab. 33,
no. 1, 160-187, 2001.

\bibitem{FEL} Feller W.,\emph{\ An Introduction Probability Theory and its
Applications}, vol.2 (2nd ed.), Wiley, New York, 1971.

\bibitem{GRA} Gradshteyn I.S., Ryzhik I.M., \emph{Tables of Integrals,
Series and Products, }5th edition, Academic Press, London, 1994.

\bibitem{JAM} Jameson G.J., The incomplete gamma functions, \emph{The Math.
Gazette}, 100 (548), 2016, 298-306.

\bibitem{KIL} Kilbas A.A., Srivastava H.M., Trujillo J.J., \emph{Theory and
Applications of Fractional Differential Equations}, vol. 204 of
North-Holland Mathematics Studies, Elsevier Science B.V., Amsterdam, 2006.

\bibitem{KUM} Kumar A., Gajda J., Wy\l omanska A. Po\l oczanski R.,
Fractional Brownian motion delayed by tempered and inverse tempered stable
subordinators, \emph{Methodol. Comput. Appl. Probab., } 21, (2019), 185--202.

\bibitem{KUM2} Kumar A., Wy\l omanska A., Po\l oczanski R., Sundar S.,
Fractional Brownian motion time-changed by gamma and inverse gamma process,
\emph{Physica A: Statistical Mechanics and its Applications,}\textit{\ (}%
2017), 468 (15), 648-667.

\bibitem{KUM3} Kumar A., Vellaisamy, P., Inverse tempered stable
subordinators, \emph{Stat. Prob. Lett.}\textit{\ }103 (2015), 134-141.

\bibitem{MAN} Mandelbrot BB, Ness JWV., Fractional Brownian motions,
fractional noises and applications. \textit{SIAM Rev} 10(4):422--437.

\bibitem{MAT} Mathai A.M.,R.K., Saxena, Haubold H.J., \emph{The H-functions:
Theory and Applications, }Springer, New York, 2010.

\bibitem{MATS} Matsui M., Pawlas Z., Fractional absolute moments of heavy
tailed distributions, \emph{Braz. J. Probab. Stat.,} vol.30, 2 (2016),
272-298.

\bibitem{MEE} Meerschaert M., Sikorskii A., \emph{Stochastic Models for
Fractional Calculus, }43,\emph{\ }De Gruyter Studies in Mathematics Series,
Berlin, 2012.

\bibitem{MET} Metzler J.R., Klafter J., The random walk:
s guide to anomalous diffusion: a fractional dynamics approach, Phys. Rep.,
339 (2000), 1-77.

\bibitem{MIJ} Mijena J.B.,Correlation structure of time-changed fractional
Brownian motion, \emph{Mathematics, }2014, n.117125740.

\bibitem{RIC} Orsingher E., Ricciuti C., Toaldo B., Time-inhomogeneous
jump processes and variable order operators, \emph{Potential Analysis}, 45, 2016,
 435-461.

\bibitem{SAM} Samorodnitsky G., Taqqu M., \emph{Stable Non-Gaussian Random
Processes: Stochastic Models with Infinite Variance, }Chapman and all, New
York, (1994).

\bibitem{SAT} Sato K.I., \emph{L\'{e}vy Processes and Infinitely Divisible
Distributions}, Cambridge Studies in Adv. Math. 68, Cambridge, 1999.

\bibitem{SCH} Schilling R.L., Song R. , Vondracek Z., \emph{Bernstein
Functions: Theory and Applications}, 37, De Gruyter Studies in Mathematics
Series, Berlin, 2010.

\bibitem{WOI} Wojdylo J.,On the coefficients that arise from Laplace's
method, \emph{Journal of Computational and Applied Mathematics, }196, 1,
(2006), 241-266.

\bibitem{WON} Wong R, \emph{Asymptotic Approximations of Integrals,} 2001%
\emph{, }SIAM ed., Philadelphia.

\bibitem{WOL} Wolfe S. J., On moments of probability distribution functions,
In: \emph{Fractional Calculus and Its Applications}, B. Ross (ed.), Lect.
Notes in Math. 457, Springer, Berlin, 1975, 306--316.
\end{thebibliography}
\end{document}